\documentclass[10pt,a4paper]{article}
\usepackage[utf8]{inputenc}
\usepackage{amsmath}
\usepackage{amsfonts}
\usepackage{amssymb}
\usepackage{amsthm}
\usepackage{mathtools}
\usepackage{setspace} 
\usepackage[left=2cm,right=2cm,top=2cm,bottom=2cm]{geometry}

\newtheorem{lemm}{Lemma}
\newtheorem{theo}[lemm]{Theorem}
\newtheorem{prop}[lemm]{Proposition}

\newtheorem{coro}[lemm]{Corollary}

\newcommand{\f}[1]{\widehat{#1}}
\newcommand{\fxy}[1]{\mathcal{F}_{xy} \left( {#1}\right)}
\newcommand{\ft}[1]{\mathcal{F}_t \left( {#1}\right)}
\newcommand{\bi}{\mathbf{i}}
\newcommand{\bj}{\mathbf{j}}
\renewcommand{\H}{\mathcal{H}}

\title{Polynomial Growth of high Sobolev norms of solutions to the Zakharov-Kuznetsov equation }
\date{\today}
\author{Raphaël Côte and Frédéric Valet}
\AtEndDocument{\bigskip{ \normalsize
 \begin{center}
  \textsc{Raphaël Côte} \small  \par 
  Université de Strasbourg \par
  CNRS, IRMA UMR 7501 \par
  F-67000 Strasbourg, France \par  
  \texttt{cote@math.unistra.fr} \par
  \medskip 
  \normalsize \textsc{Frédéric Valet} \par \small 
  Université de Strasbourg \par
  CNRS, IRMA UMR 7501 \par
  F-67000 Strasbourg, France \par  
  \texttt{valet@math.unistra.fr}
 \end{center}
}}

\spacing{1.15}

\begin{document}

\maketitle
\let\thefootnote\relax
\let\thefootnote\relax

\begin{abstract}
We consider the Zakharov-Kuznetsov equation  \eqref{ZK} in space dimension 2. Solutions $u$ with initial data $u_0 \in H^s$ are known to be global if $s \ge 1$. We prove that for any  integer $s \ge 2$, $\| u(t) \|_{H^s}$ grows at most polynomially in $t$ for large times $t$. This result is related to wave turbulence and how a solution of \eqref{ZK} can move energy to high frequencies. 

It is inspired by analoguous results by Staffilani \cite{Sta} on the non linear Schrödinger and Korteweg-de-Vries equation. The main ingredients are adequate bilinear estimates in the context of Bourgain's spaces and a careful study of the variation of the $H^s$ norm.

\end{abstract}

\vspace{1cm}

We are interested in the 2D Zakharov-Kuznetsov equation
\begin{align}\tag{ZK}\label{ZK}
\begin{cases}
\partial_t u + \partial_x \left( \Delta u + u^2 \right)=0, \quad u:  I_t \times \mathbb{R}^2_{x,y} \rightarrow \mathbb{R} \\
u(0) = u_0 \in H^s(\mathbb{R}^2)
\end{cases}
\end{align}
where $\Delta=\partial_x^2 + \partial_y^2$ is the laplacian on $\mathbb{R}^2$, the time interval $I_t \ni 0$ is the interval of existence. This equation has been studied to model propagation, in magnetized plasma, of non-linear ion-acoustic waves, see \cite{KZ}. \eqref{ZK} has also been derived from the Euler-Poisson system in dimension $d=2$ and $d=3$ by Lannes, Linares and Saut in \cite{LLS}, and from the Vlasov-Poisson equation by Han-Kwan \cite{Han13}. 

This equation naturally enjoys some conserved quantities (at least formally) like the mass and the energy:
\begin{align*}
M(u) := \frac{1}{2}\int_{\mathbb{R}^2} u^2(t,x,y) dxdy  , \quad E(u) := \frac{1}{2} \int \left( \vert \nabla u(t,x,y) \vert^2- \frac{1}{3} u(t,x,y) ^3 \right)dxdy.
\end{align*}
The Cauchy problem of (\ref{ZK}) has been first studied by Faminskii in \cite{Fam}, who showed local and global well-posedness in $H^1(\mathbb{R}^2)$. This local well-posedness has then been improved in \cite{LP11}, and then by Molinet and Pilod in \cite{MP} and independently by Grünrock and Herr in \cite{GH} who proved the local well-posedness in $H^s$ for $s\ge\frac{1}{2}$. Recently, Kinoshita \cite{Kin} improved this and prove local well-posedness in $H^s$ for any $s > -\frac{1}{4}$, which is the sharp exponent. 

For $s \ge 1$, the mass $M(u)$ and the energy $E(u)$ are well defined and preserved by the flow.
Due to a Gagliardo-Nirenberg inequality, there exist a universal constant $C$ such that
\begin{equation} \label{est:H1_ME}
\forall v \in H^1(\mathbb R^2), \quad \| v \|_{H^1}^2 \le C(1+ E(v) + M(v)^2).
\end{equation}
As a consequence \eqref{ZK} is globally well-posed, and the $H^1$ norm $\| u(t) \|_{H^1}$ remains bounded for all times.

One can naturally ask what happens for $\| u(t) \|_{H^s}$. If for some $s >1$, $\| u(t) \|_{H^s} \to +\infty$, one speaks of \emph{energy cascade phenomenon}, which means that some energy move from low frequencies to high frequencies: it is an important aspect of out of equilibrium dynamics, predicted and studied on a number of nonlinear dispersive models,  under the name of \emph{wave turbulence} in the Physics literature. To the contrary, for integrable systems, one expects that all $\| u(t) \|_{H^s}$ remain bounded.

One should note that the proof of local well-posedness often allows to give exponential (or double exponential) bounds on $\| u(t) \|_{H^s}$. On the other hand, constructing a solution which displays an energy cascade phenomenon is very delicate, we refer to the work by Hani, Pausader, Tzvetkov and Visciglia \cite{HPTV15} for an example in the context of Schrödinger type equations. 

\bigskip

In this article, we prove that the $H^s$-norm of a solution $u$ of  \eqref{ZK} equation grows at most polynomially for large times. 

\begin{theo} \label{th1}
Let $s\ge 2$ be an integer and $u_0 \in H^s(\mathbb{R}^2)$. Denote $u$ the solution of \eqref{ZK} with initial data $u_0$ and $A = \sup_{t \ge 0} \| u(t) \|_{H^1}$. Then $u \in \mathcal C(\mathbb R, H^s)$ and for any $\beta> \frac{s-1}{2}$,  there exist a constant $C = C(s,\beta,A)$, such that
\begin{align} \label{control_poly}
\forall t \in \mathbb R, \quad \| u(t) \|_{H^s} \le C (1+|t|)^\beta( 1+ \| u_0 \|_{H^s}),
\end{align}
\end{theo}

The start of the study of the polynomial growth of norms is due to Staffilani \cite{Sta} in the context of non linear Schrödinger and Kor\-te\-weg-de Vries type equations, with ideas of Bourgain \cite{Bou93} and \cite{Bou93bis}. It was later extended to many situations: let us refer to Sohinger \cite{Soh11} for the Schrödinger and Hartree equations, or Planchon, Tzvekov, Visciglia \cite{PTV17} for the Schrödinger equation on a manifold, and the references therein.

\bigskip

The method of proof in \cite{Sta} is to refine the local well-posedeness statement. Usually, what is proved is that given $u_0$, there exist $C,T >0$ such that one can construct a solution $u$ on $[-T,T]$ and such that
\begin{equation} \label{eq:lwp}
\forall t \in [-T,T], \quad \| u(t) \|_{H^s} \le C \| u(0) \|_{H^s}.
\end{equation}
If $C$ and $T$ only depend on $\| u_0 \|_{H^1}$, one can use an iteration argument and obtain global well-posedness in $H^s$ with an exponential bound on the $H^s$ norm. The heart of the method lies in the observation that, for $H^s$-norms (with $s>1$), a similar bound can be obtained, but with a slightly better exponent on the $H^s$ norm on the right-hand side. More precisely, there exists $\epsilon >0$ and two functions $C,T:  [0,+\infty) \to [0,+\infty)$ ($C$ increasing, and $T$ decreasing) such that for any solution $u$ of \eqref{ZK}, and for all $t_0 \in \mathbb{R}$,
\begin{equation} \label{eq:imp_lwp}
\forall t \in [t_0 - T(\| u(t_0) \|_{H^1}), t_0 + T(\| u(t_0) \|_{H^1})], \quad \| u(t) \|_{H^s} \le \| u(t_0) \|_{H^s} + C(\| u(t_0) \|_{H^1}) (1+\| u(t_0) \|_{H^s}^{1-\epsilon}).
\end{equation}
The key point is that the time of existence $T$ and the amplification factor $C$ do only depend of the $H^1$ norm of $u$, which is known to be uniformly bounded for all times, and so will essentially not depend on $t_0$.

With the above improved bound \eqref{eq:imp_lwp} in hand, one can conclude the proof of Theorem \ref{th1} by a straightfoward iteration argument, by working on the time intervals $[k T(A), (k+1) T(A)]$ for $k \in \mathbb{N}$ (see Lemma \ref{lem13} for details).

To derive estimate \eqref{eq:lwp}, one considers the integral formulation of the equation, and proves that the Duhamel term enjoys a better behavior than the linear term; we use Strichartz estimates, a bilinear estimate due to Molinet and Pilod \cite{MP}, and its tame version. With \eqref{eq:lwp} in hand, one can then look for the improved version \eqref{eq:imp_lwp}. For this, \cite{Sta} makes use of bilinear estimates due to Kenig-Ponce-Vega in \cite{KPV96} in Bourgain spaces $X^{s,b}$, \cite{Soh11} relies on the $I$-method (first developed in \cite{CKSTT}), and \cite{PTV17} considers high order modified functionals of energy type. In this paper, we will follow the method in \cite{Sta} and work in suitable Bourgain spaces. In order to derive \eqref{eq:imp_lwp}, we study the variation of the $H^s$ norm. Compared with (KdV) for example, one has to be extra careful on how derivatives fall on each factor, because the algebraic structure of \eqref{ZK} is not as strong as that of (KdV), and the dispersion effects are weaker for \eqref{ZK} than for (KdV): as it can be seen from the bilinear estimates in \cite{MP} which require positive index of regularity (and the fact that the local well posedness results are not optimal). To suitably bound the worst terms, we prove a new bilinear estimate (in Proposition \ref{bilinear2prop}) for negative regularity, which is the main new technical tool of this paper.

\bigskip

The article is organized as follows. We recall in section 1 Strichartz estimates and embeddings between Bourgain's spaces and Sobolev spaces for \eqref{ZK}, in order to fix notations. In section 2, we collect and prove the bilinear estimates required to deal with the non linear term. In section 3, we prove a local well posedness result with time of existence depending only on $\| u(0) \|_{H^1}$, and then we conclude the proof of Theorem \ref{th1}.

\section{Notations and basic facts}

\subsection{Notations}

For $x\in \mathbb{C}$, $\vert x \vert$ denotes the module of $x$, and for $(x,y)\in \mathbb{R}^2$, we use the norm
\[ \vert (x,y) \vert := \sqrt{3x^2 +y^2}, \]
(it is anistropic but convenient for our purposes). We denote by $\f{u}$ the Fourier transform with respect to the \emph{space and time} variables, and the inverse Fourier transform by $u^\vee$ . The Fourier transform with respect to $t$ or the space variables only will be respectively denoted $\ft{u}$ or $\fxy{u}$. The derivatives will be denoted $\partial_t$, $\partial_x$ or $\partial_y$.

For a general function $P(t,x,y)$, we define the associated  self-adjoint and positive differential operator by:
\begin{align*}
P(D) f(t,x,y):= \int_{\mathbb{R}^3} \left\vert P(i\tau, i\xi, i\eta) \right\vert e^{i\left( \tau t  + \xi x + \eta y \right)} \f{f}(\tau, \xi,\eta) d\tau d\xi d \eta.
\end{align*}
For instance, we will regularly use the differential operators $D_x$ for $P=x$, $D_y$ for $P=y$ and $D_t$ for $P=t$ acting on frequencies:
\begin{align*}
\fxy{D_x f}(\xi,\eta)= \vert \xi \vert \fxy{f}(\xi,\eta), \quad \fxy{D_y f}(\xi,\eta)= \vert \eta \vert \fxy{f}(\xi,\eta), \quad \ft{D_t f}(\tau)= \vert \tau \vert \ft{f}(\tau).
\end{align*}
We define the derivation operator $S(D)$ with $S(t,x,y)=\langle \vert (x,y) \vert \rangle$ (where  $\langle a \rangle:=\left( 1+ a^2\right)^\frac{1}{2}$ is the Japanese bracket), to the power $s$ to define the Sobolev space $H^s(\mathbb{R}^2)$:
\begin{align}\label{S}
\fxy{ S(D)^s f} (x,y) := \langle \vert (\xi,\eta) \vert \rangle^s \fxy{f}(\xi,\eta), \quad \| f \|_{H^s} := \| S(D)^s f \|_{L^2}. 
\end{align}

We will manipulate multi-indices at many places: they will always be denoted by a bold letter $\bi=(i_1,i_2)$, with $i_1$ and $i_2$ in $\mathbb{N}$. Then we denote $D^\bi$ the differential operator with $i_1$ derivatives on the $x$ variable, and $i_2$ on the $y$ variable:
\begin{align*}
D^\bi f(x,y):= D^{i_1}_x D^{i_2}_y f (x,y)= \int_{\mathbb{R}^2} \vert \xi \vert^{i_1} \vert \eta \vert^{i_2} e^{i\left( \xi x + \eta y \right)} \fxy{f}(\xi,\eta)d\xi d\eta.
\end{align*}

The length of a multi-index $\bi = (i_1,i_2)$ will be denoted by $\vert \bi \vert := i_1+i_2$. We also use a partial order on multi-indices:
\begin{align*}
\bj \le \bi \quad \text{ if } \quad j_1 \le i_1 \text{ and } j_2 \le i_2, \quad \text{ and } \quad \bj < \bi \quad \text{ if } \quad (j_1+1, j_2) \le \bi \text{ or } (j_1, j_2 +1) \le \bi.
\end{align*}

The linear part of the equation (\ref{ZK})  $\partial_{t} u + \partial_x \Delta u =0$  induces a semigroup denoted $W:\mathbb{R}_t\rightarrow \mathcal{L}(L^2(\mathbb{R}^2))$, defined by :
\begin{align*}
\fxy{ W(t) u}(\xi,\mu):= e^{i t\omega(\xi, \mu)}\fxy{u}(t,\xi,\mu) \quad \text{ with } \quad \omega(\xi,\mu)=\xi \left( \xi^2 + \mu^2\right).
\end{align*}
Similarly, this linear flow can be represented in space-time, which motivates the introduction of the function $F(t,x,y) :=  \langle \vert t+x(x^2+y^2) \vert\rangle$ and the operator $F(D)$ to the power $b\in \mathbb{R}$:
\begin{align}\label{F}
\f{F(D)^b f}(\tau, \xi, \eta) := \langle \tau- \omega( \xi,\eta) \rangle^b \f{f}(\tau, \xi, \eta).
\end{align}
We then introduce the adapted Bourgain spaces. Given two indices $s,b \in \mathbb{R}$, the Bourgain space $X^{s,b}$ takes into account first the norm $H^s$ in space of a function $u(t,x,y)$, and secondly how far away ($H^b$ in time) the solution is from the linear flow: the associated norm is 
\begin{align*}
\| u \|_{X^{s,b}}^2:= \| S(D)^s F(D)^b u \|_{L^2}^2= \int_{\mathbb{R}^3} \langle \vert (\xi, \mu)\vert \rangle ^{2s} \langle \tau- \omega(\xi, \mu)\rangle^{2 b} \left\vert \f{u}(\tau, \xi, \mu) \right\vert ^2 d\tau d\xi d\mu,
\end{align*}
where  $S$ and $F$ are  defined above in (\ref{S}) and (\ref{F}). $X^{s,b}$ is the completion of the Schwartz functions of $\mathbb R_t \times \mathbb R^2_{x,y}$ for this norm. 

Because of the conserved quantities, all the solutions  $u$ of (\ref{ZK}) (even small) will not have finite $X^{s,b}$-norm. This is why we will consider solution to a version of the integral formulation of \eqref{ZK} which is truncated in time, by the use of a smooth cut-off function. Let the non-negative smooth function $\phi$ equal to 1 on $[0,1]$, and $0$ outside $[-1,2]$, and the $T$-dilated functions $\phi_T(t)= \phi\left( \frac{t}{T}\right)$. We will always assume in the following that $T \le 1$. The truncated Duhamel operator is given by the following formula:
\begin{align}\label{Psi_1}
\Psi(w)(t):= \phi_1(t) W(t)u_0 + \phi_T(t) \int_0^t W(t-t')\phi_T(t')^2\partial_x \left( w^2 \right)(t')dt'.
\end{align}
 Observe that if we can construct a fixed point $w$ such that $\Psi(w) = w$, then $w$ is a solution to \eqref{ZK} on the interval $[-T,T]$ (with initial data $w(0) = u_0$).

In order to prove the various (bilinear) estimates, we will need to work with a Littlewood-Paley decomposition. Let a positive and non increasing function $\chi_0$ in $\mathcal{C}^\infty_0([0,+\infty))$, such that
\[ \chi_0(r) =1 \quad \text{if } r \le \frac{5}{4}, \quad \text{and} \quad \chi_0(r) = 0 \quad \text{if } r \ge \frac{8}{5}. \]
 Let 
\begin{align*}
\chi(r)= \chi_0\left(\frac{r}{2}\right)-\chi_0 (r) \quad \text{and } \forall j \in \mathbb{N}, \quad \chi_j(r) := \chi \left( \frac{r}{2^j} \right).
\end{align*}
We will use the following truncation operators, related to space and to the linear flow, recalling that $\vert(\xi, \mu) \vert = \sqrt{3\xi^2 + \mu^2}$:
\begin{align} \label{def_PQ}
P_N(u):= \left( \chi_N(\langle \vert( \xi,\mu) \vert \rangle) \f{u}(\tau,\xi,\mu) \right)^\vee \text{ and } Q_L u := \left( \chi_L \left( \langle \tau-\omega(\xi,\mu) \rangle \right) \f{u}(\tau,\xi,\mu) \right)^\vee. 
\end{align}

For $b \in \mathbb R$, we will sometimes denote $b^+$ for any number $b+\epsilon$ with $\epsilon >0$ small. Finally, given two functions $f$ and $g$ (depending on space and/or time), we denote
\begin{align*}
f \lesssim g \quad \text{if} \quad \exists C>0, \quad f \le C g,
\end{align*}
for an absolute implied constant $C$, independent of $f$ and $g$ (unless otherwise specified), which can change from one line to the next.
Sometimes $\lesssim$ is used for quantities involving a $b^+$: in that case, the implied constant may depend on $\epsilon>0$.

\subsection{Strichartz estimates and Bourgain spaces}

Now we recall some lemmas which are the heart the following proofs, to begin with the Strichartz estimates for the Zakharov-Kuznetsov equation in dimension 2:

\begin{prop}[{\cite[Proposition 3.1]{LS} and \cite[Proposition 2.4]{LP09}}]
Let $\epsilon \in \left[0, \frac{1}{2} \right) $, $\theta \in [0,1]$ and define
 \[ p=\frac{2}{1-\theta}, \quad \text{and} \quad q= \frac{6}{\theta(2+ \epsilon)}. \]
Let $f$ be a function defined on $\mathbb{R}^2$, and $g$ be defined on $\mathbb{R}_t \times \mathbb{R}^2 $.  Then the following estimates hold:
\begin{align} \label{Strichartz}
\left\|D_x^{\frac{\epsilon \theta}{2}}W(t)f \right\|_{L^q_t L^p} & \lesssim \|f \|_{L^2}, \\
\left\| D_x^{\theta \epsilon} \int W(t-t')g(t')dt' \right\|_{L^q_t L^p} & \lesssim \| g \|_{L^{q'}_tL^{p'}}, \nonumber \\
\left\|D_x^{\theta \epsilon} \int W(t) g(t) dt \right\|_{L^2} & \lesssim \|g \|_{L_t^{q'}L^{p'}}. \nonumber
\end{align}
\end{prop}

In particular, the first estimate yields an embedding in the context of Bourgain spaces:

\begin{coro}[{\cite[Corollary 3.2]{MP}}]
There hold
\begin{align} \label{embedding}
\| u \|_{L^4_{txy}} \lesssim \| u \|_{X^{0,\frac{5}{12}^+}},
\end{align}
\end{coro}

Furthermore, \eqref{Strichartz} (up to the use of Cauchy-Schwarz inequality) gives us an other embedding: if  $u \in X^{s,\frac{1}{2}^+}$, then uniformly in $t \in \mathbb R$,
\begin{align}\label{embedding2}
\| u (t) \|_{H^s} \lesssim \| u \|_{X^{s,\frac{1}{2}^+}}.
\end{align}

We now focus on the terms of the Duhamel formula (\ref{Psi_1}). We give the following two linear estimates in Bourgain spaces: the first estimate is useful when dealing with the linear term; the second for the Duhamel term.

\begin{lemm}[{\cite[Lemmas 3.1 and 3.2]{G}}]
\label{lemma5}
Fix $T \le 1$.

Let $b \ge 0$ and $f\in H^s(\mathbb{R}^2)$. Then:
\begin{align}\label{inclusion}
\| \phi_T(t) W(t) f\|_{X^{s,b}} \lesssim T^{\frac{1}{2}-b}\| f \|_{H^s}.
\end{align}

Let $- \frac{1}{2} < b' < 0< b \le 1+b'$, and $g\in X^{s,b'}$. Then
\begin{align}\label{autre_inclusion}
\left\| \phi_T(t) \int_0^t W(t-t')g(t')dt' \right\|_{X^{s,b}} \lesssim T^{1-b+b'} \| g \|_{X^{s,b'}}.
\end{align}
\end{lemm}

\begin{proof}
A proof of these estimates in the context of the Korteweg-de Vries equation is done in the review article by Ginibre \cite[Lemma 3.1 and 3.2]{G}. For the convenience of the reader, we provide below a full proof for \eqref{ZK}.

The idea for estimate (\ref{inclusion}) is to separate the variables. In particular, we detect well the need of a fixed time interval.
\begin{align*}
\MoveEqLeft
\| \phi_T W(t) f \|_{X^{s,b}}^2 = \int_{\mathbb{R}^3} \langle \vert(\xi,\mu)\vert \rangle^{2s} \langle \tau-\omega(\xi, \mu) \rangle^{2b} \left\vert \int e^{-i(t\tau-t\omega+x\xi+y\mu)} \left( \phi_T f \right)dxdydt \right\vert^2 d\xi d\mu d\tau \\
	& =\int \langle \vert(\xi,\mu)\vert \rangle^{2s} \langle \tau' \rangle^{2b} \left\vert \ft{\phi_T}(\tau')\fxy{f}(\xi,\mu)\right\vert^2 d\xi d\mu d\tau  = \| \phi_T\|_{H^b_t} ^2 \| f \|_{H^s_{x,y}}^2.
\end{align*}
Recalling that $\langle a \rangle \lesssim  1+ \vert a \vert$, we obtain $\| \phi_T\|_{H_t^b}^2\le T^{1-2b} \|\phi\|_{H_t^b}^2$, and the first estimate is proved.

For estimate (\ref{autre_inclusion}): notice that the sum over $t$ in time is in fact a division by $\tau$ in frequency. Denote $\omega:=\omega(\xi,\mu)$, we split the domain whether $T|\tau - \omega| \lessgtr 1$, and obtain:
\begin{align*}
\MoveEqLeft
\fxy{\phi_T(t) \int_0^t W(t-t') g(t') dt'}= \phi_T(t) e^{it\omega} \int \frac{e^{it(\tau-\omega)}-1}{i(\tau-\omega)}\f{g}(\tau) d\tau \\
	& = \phi_T(t) e^{it\omega} \int_{T\vert \tau -\omega \vert \ge 1} \frac{i}{\tau-\omega} \f{g}(\tau)d\tau + \phi_T(t) e^{it\omega} \sum_{k\ge 1}\frac{t^k}{k!} \int_{T\vert \tau -\omega \vert \le 1} \left(i(\tau-\omega)\right)^{k-1} \f{g}(\tau)d\tau \\
		& \quad +\phi_T(t)e^{it\omega} \int_{T\vert \tau -\omega \vert \ge 1} \frac{e^{it(\tau-\omega)}}{i(\tau-\omega)}\f{g}(\tau)d\tau \\
		& = I+II+III.
\end{align*}
For $I$, we proceed as in the proof of (\ref{inclusion}). First, by separation of variables:
\begin{align*}
\ft{I}(\theta)= \ft{\phi_T}(\theta-\omega) \int_{T \vert \tau- \omega \vert \ge 1} \frac{i}{\tau-\omega} \f{g}(\tau) d\tau.
\end{align*}
Now, observe that $b' > -\frac{1}{2}$, and the change of variable $\theta\rightarrow \theta+\omega$ gives:
\begin{align*}
\MoveEqLeft 
\int \langle \vert (\xi,\mu)\vert \rangle^{2s} \langle \theta-\omega \rangle^{2b}\vert \ft{I}(\theta,\xi,\mu)\vert^2 d\theta d\xi d\mu \\
	& \le \|\phi_T\|_{H^b_t}^2 \int \langle \vert( \xi,\mu)\vert \rangle^{2s} \left(\int \langle \tau-\omega \rangle^{2b'} \vert \f{g} \vert^2 d\tau \right) \left( \int_{T\vert \tau-\omega \vert \ge 1} \frac{1}{(\tau-\omega)^{2+2b'}}d\tau \right)d\xi d\mu \\
	& \lesssim T^{1-2b}\| g \|_{X^{s,b'}}^2 T^{1+2b'} = \left( T^{1-b+b'} \| g\|_{X^{s,b'}}\right)^2.
\end{align*}
Similarly, we compute the time Fourier transform of the term $II$:
\begin{align*}
\ft{II}(\theta) = \sum_{k\ge 1} \frac{1}{k!} \ft{t^k \phi_T}(\theta-\omega) \int_{T\vert \tau-\omega \vert \le 1} \left( i (\tau-\omega)\right)^{k-1} \f{g}(\tau)d\tau.
\end{align*}
With similar techniques we used for the term $I$:
\begin{align*}
\MoveEqLeft 
\int \langle \vert (\xi,\mu)\vert \rangle^{2s} \langle \theta-\omega \rangle^{2b}\vert \ft{II}(\theta,\xi,\mu)\vert^2 d\theta d\xi d\mu \\
	& \le \sum_{k\ge 1} \frac{1}{k!} \|t^k \phi_T(t)\|_{H^b_t}^2 \int \langle \vert(\xi,\mu)\vert \rangle^{2s} \left(\int \langle \tau-\omega \rangle^{2b'} \vert \f{g}(\tau)\vert^2 d\tau \right) \left( \int_{T\vert \tau-\omega \vert \le 1} \frac{\vert \tau-\omega \vert^{2k-2}}{ \langle \tau-\omega \rangle^{2b'}}d\tau \right) d\xi d\mu \\
	& \lesssim \sum_{k\ge 1} \frac{1}{k!} T^{1-2b+2k} \| g \|_{X^{s,b'}}^2T^{1-2k+2b'} \le \left( T^{1-b+b'} \| g\|_{X^{s,b'}}\right)^2.
\end{align*}
We finally need to bound $III$, in which the time variable $t$ appears inside the integral on $\tau$. The integral :
\begin{align*}
J(t) :=\int_{T \vert \tau -\omega \vert \ge 1}  \frac{e^{it(\tau-\omega)}}{i(\tau-\omega)}\f{g}(\tau) d\tau
\end{align*} can be seen as the inverse of a Fourier transform:
\begin{align*}
\int \langle \theta-\omega \rangle^{2b} \left\vert \int e^{-it \theta}J(t) dt\right\vert ^2 d\theta = \int_{T \vert \theta-\omega\vert \ge 1} \frac{\langle \theta -\omega \rangle^{2b}}{(\theta-\omega)^2} \vert \f{g}(\theta-\omega)\vert^2 d\theta  \lesssim \left( T^{1-b+b'} \|\f{g}(\xi,\mu)\|_{H^{b'}_t} \right)^2.
\end{align*}

Similarly, for the $L^2$-norm:
\begin{align*}
\int \left\vert \int e^{-it \theta} J(t) dt\right\vert ^2 d\theta \lesssim \left( T^{1+b'} \|\f{g}(\xi,\mu)\|_{H^{b'}_t} \right)^2.
\end{align*}
We can now compute the norm of the term $III$:
\begin{align*}
\MoveEqLeft
\int \langle \vert (\xi,\mu)\vert \rangle^{2s} \langle \theta-\omega \rangle^{2b}\vert \ft{III}(\theta,\xi,\mu)\vert^2 d\theta d\xi d\mu \\
	& =\int \langle \vert (\xi,\mu) \vert \rangle^{2s} \left( \int \langle \theta-\omega \rangle^{2b} \vert \ft{e^{it\omega}\phi_T(t)} * \ft{J}(\theta) \vert^2 d\theta \right) d\xi d\mu \\
	& \lesssim \| \langle  \tau \rangle^{2b} \ft{\phi_T} \|_{L^1_t} \| J\|_{L_t^2H^s}+ \| \ft{\phi_T} \|_{L^1_t} \| J\|_{X^{s,b}} \lesssim \left( T^{1-b+b'} \| g \|_{X^{s,b'}}\right)^2. \qedhere
\end{align*}

\end{proof}

Observe that $X^{s,b}$ are embedded in one another as $b$ decreases, even after truncation in time.

\begin{lemm} 
For all $b'<b$ satisfying $0 < b-b' < \frac{1}{2}$, and any function $u$ in $X^{s,b}$:
\begin{align}\label{differentb}
\| \phi_T(t) u \|_{X^{s,b'}} \lesssim T^{b-b'} \| \phi_T(t) u \|_{X^{s,b}}.
\end{align}
\end{lemm}
We emphasize that the implicit constants above do not depend on the parameters $b$, $b'$ and $T$.

\begin{proof} By choosing $p:= \frac{1}{1+2(b'-b)}$ and $p':= \frac{1}{2(b-b')}$:
\begin{align*}
\| \phi_T u \|_{X^{s,b'}}^2 
	& = \int \langle \vert ( \xi, \mu) \vert \rangle^{2s} \| \langle \tau - \omega \rangle^{b'}  \f{\phi_T u } \|_{L^2}^2 d(\xi,\mu) \\
	& = \int \langle \vert ( \xi, \mu) \vert \rangle^{2s} \int \langle \tau \rangle^{2b'} \vert \f{W(t)\phi_T u } \vert^2(\tau ) d \tau d(\xi,\mu) \\
	& = \int \langle \vert ( \xi, \mu) \vert \rangle^{2s} \int \vert \langle D_t \rangle ^{2b'} W(t)\fxy{\phi_T u} \vert^2 dt d (\xi,\mu) \\
	& \le \int \langle \vert ( \xi, \mu) \vert \rangle^{2s} \left( \int \vert \langle D_t \rangle ^{b'} W(t)\fxy{\phi_T u} \vert^{2p} dt\right)^{\frac{1}{p}} \left( \int_{-T \le t \le 2T} dt \right)^{\frac{1}{p'}} d (\xi,\mu).
\end{align*}

Then by a Gagliardo-Nirenberg-Sobolev embedding $\| v(t) \|_{L^{2p}_t} \le \|D_t^{\frac{1}{2p'}} v(t) \|_{L^2_t}$:
\begin{align*}
\| \phi_T u \|_{X^{s,b'}}^2 
	& \lesssim T^{2(b-b')} \int \langle \vert ( \xi, \mu) \vert \rangle^{2s} \int \vert \langle D_t \rangle ^{2b} W(t) \fxy{\phi_T u} \vert^{2} dt d(\xi,\mu) \lesssim T^{2(b-b')} \| \phi_Tu \|_{X^{s,b}}^2. \qedhere
\end{align*}
\end{proof}

In order to prove the bilinear estimates in the next section, we will make use of two preliminary lemmas already stated in Molinet Pilod \cite{MP}. The first one is the following:

\begin{lemm}[{\cite[Proposition 3.5]{MP}}] 
Consider the polynomial $K(x,y):= 3x^2- y^2$. For all $\phi \in L^2(\mathbb{R}^2)$,
\begin{align} \label{restriction_theorem2}
\left\| K(D) ^{1/8} e^{-t \partial_x \Delta} \phi \right\|_{L^4_{txy}} \lesssim \| \phi \|_{L^2_{xy}},
\end{align}
and for all $u \in X^{0,\frac{1}{2}^+}$,
\begin{align}\label{restriction_theorem}
\left\| K(D)^{1/8} e^{t \partial_x \Delta} u \right\|_{L^4} \lesssim \| u \|_{X^{0,\frac{1}{2}^+}}.
\end{align}
\end{lemm}

\begin{proof}
This result is in fact a direct consequence of the following optimal $L^4$ restriction estimate for homogenous polynomial hypersurface of degree $d\ge 2$ in $\mathbb{R}^3$ proved by Carbery, Kenig and Zisler \cite{CKZ}, which goes as follows. 

Let $\Omega \in \mathbb R[X,Y]$ be a homogeneous polynomial of degree $d \ge 2$ and let $\Gamma(\xi, \mu) = (\xi, \mu \Omega(\xi,\mu))$. Denote $K_\Omega(\xi, \mu) = |\det \mathop{\text{Hess}} \Omega(\xi, \mu)|$. Then there exist a constant $C>0$ such that
\begin{equation} \label{est:restriction_CKZ}
\forall f \in L^{4/3}(\mathbb R^3), \quad \left( \int_{\mathbb R^2} | \hat f(\Gamma(\xi,\mu))|^2 K_\Omega(\xi, \mu)^{1/4} d\xi d\mu \right)^{1/2} \le C \| f \|_{L^{4/3}}.
\end{equation}
The symbol associated to $e^{-t \partial_x \Delta}$ is $\omega(\xi,\mu) = \xi \left( \xi^2 + \mu^2 \right)$, whose Hessian is
\[ \det \mathop{\text{Hess}} \omega(\xi, \mu) = 12 \xi^2 - 4 \mu^2 = 4 K(\xi, \mu). \]
We can then apply \eqref{est:restriction_CKZ} to $K$, and by a direct duality argument, we derive \eqref{restriction_theorem2} and \eqref{restriction_theorem}.
\end{proof}

One can view this result as a $1/4$ gain of derivative when compared to the Strichartz estimate \eqref{Strichartz} (we refer to \cite[Remark 3.1]{MP} for further details).

The second lemma reveals where the gain of regularity occurs in the dispersion relation of \eqref{ZK}. It is specific to dimension 3, and so well suited to the study in Bourgain spaces of \eqref{ZK}  in 2 space dimensions. We recall that the truncation operators $P_N$ and $Q_L$ were defined in \eqref{def_PQ}.

\begin{lemm}[{\cite[Proposition 3.6]{MP}}]\label{lemma7}
Let $N_1$, $N_2$, $L_1$ and $L_2$ be four integers involved in the operators $P_N$ and $Q_L$, and two functions $u_1$ and $v_2$. Then
\begin{align}\label{measure}
\left\| \left( P_{N_1} Q_{L_1} u_1 \right) \left( P_{N_2} Q_{L_2} v_2 \right) \right\|_{L^2} \lesssim \max(L_1, L_2)^\frac{1}{2} \max(N_1,N_2) \left\| P_{N_1} Q_{L^1} u_1 \right\|_{L^2} \left\| P_{N_2} Q_{L^2} v_2 \right\|_{L^2}.
\end{align}
If furthermore $N_2 \ge 4N_1$ or $N_1 \ge 4N_2$, then
\begin{align}\label{measure2}
\left\| \left( P_{N_1} Q_{L_1} u_1 \right) \left( P_{N_2} Q_{L_2} v_2 \right) \right\|_{L^2} \lesssim \frac{\max(N_1,N_2)^\frac{1}{2}}{\min(N_1,N_2)} (L_1L_2)^\frac{1}{2} \left\| P_{N_1} Q_{L_1} u_1 \right\|_{L^2} \left\|P_{N_2} Q_{L^2} v_2 \right\|_{L^2}.
\end{align}
\end{lemm}

Observe in inequalities (\ref{measure2}) the $-\frac{1}{2}$ gain in the quotient $\frac{\max(N_1,N_2)^\frac{1}{2}}{\min(N_1,N_2)}$.

\section{Bilinear estimates}

The key bilinear estimate in \cite{MP} is the following:

\begin{prop}[{\cite[Proposition 4.1]{MP}}]
Let $s > \frac{1}{2}$. Then there exists $0<\delta<\frac{1}{4}$ such that for all $u$, $v$, two functions of $X^{s,\frac{1}{2}+ \delta}$ :
\begin{align}\label{bilinear1}
\|\partial_x (uv)\|_{X^{s,-\frac{1}{2}+2\delta}}  \lesssim \| u \|_{X^{s,\frac{1}{2}+\delta}} \| v \|_{X^{s,\frac{1}{2}+\delta}} .
\end{align}
\end{prop}

This estimates allows to gain one derivative, as required by the non linear term in \eqref{ZK}, in the context of  Bourgain spaces.
When $s$ is large, we need to slightly improve this estimate when $s>1$, to derive a tame bilinear estimate where one $s$ is exchanged for a regularity index $1$: this is important so as to make a full use of $H^1$ bounds.

\begin{coro} \label{propo8}
Let $s >1$. Then there exists a constant $0<\delta < \frac{1}{4}$, and a constant $C(s)$ such that:
\begin{align}\label{bilinear3}
\forall u, \in X^{s, \frac{1}{2}+\delta}, \quad \|\partial_x (uv)\|_{X^{s,-\frac{1}{2}+2\delta}}  \le C(s) \left( \| u \|_{X^{s,\frac{1}{2}+\delta}} \| v \|_{X^{1,\frac{1}{2}+\delta}} + \| u \|_{X^{1,\frac{1}{2}+\delta}} \| v \|_{X^{s,\frac{1}{2}+\delta}} \right).
\end{align}
\end{coro}

\begin{proof}
First, by definition of the bracket, 
\begin{align*}
\langle \vert( \xi_0, \mu_0) \vert\rangle ^{2(s-1)} \le C(s) \left(\langle \vert( \xi_1, \mu_1) \vert\rangle ^{2(s-1)}+ \langle\vert( \xi_1-\xi_0, \mu_1-\mu_0) \vert\rangle ^{2(s-1)} \right),
\end{align*}
which implies, by using the previous proposition :
\begin{align*}
\| \partial_x(uv) \|_{X^{s,-\frac{1}{2}+2\delta}}
	& \lesssim \left\| \partial_x \left( \langle \vert ( D_x,D_y) \vert \rangle^{s-1} (u) v \right) \right\|_{X^{1,-\frac{1}{2}+2\delta}} + \left\| \partial_x \left( u \langle \vert ( D_x,D_y) \vert \rangle^{s-1} (v) \right) \right\|_{X^{1,-\frac{1}{2}+2\delta}} \\
	& \lesssim \| \langle \vert ( D_x,D_y) \vert \rangle^{s-1} u \|_{X^{1,\frac{1}{2}+\delta}} \| v \|_{X^{1,\frac{1}{2}+\delta}} +  \| u \|_{X^{1,\frac{1}{2}+\delta}} \| \langle \vert ( D_x,D_y) \vert \rangle^{s-1} v \|_{X^{1,\frac{1}{2}+\delta}} \qedhere
\end{align*}
\end{proof}

We will use both estimates \eqref{bilinear1} and \eqref{bilinear3} in the next section (in particular in the fixed point result). In order to prove of Theorem \ref{th1}, we will also need another bilinear estimate in $X^{-\rho,b}$ spaces, with \emph{negative} regularity index $-\rho <0$: this gain of space derivative is crucial, and possible because we won't need the space derivative on $uv$ there.

\begin{prop}\label{bilinear2prop}
Let $\delta>0$ small ($\delta < \frac{1}{12}$), $b'= -\frac{1}{2}+2\delta$ and $b=\frac{1}{2}+\delta$. There exist a constant $C$, independent of $\delta$, such that for all $0<\rho< \frac{1}{2}-6\delta$ the following estimate holds:
\begin{align}\label{bilinear2}
\forall u,v \in X^{-\rho, b}, \quad \| uv \|_{X^{-\rho,b'}} \le C \| u \|_{X^{-\rho,b}} \|v \|_{X^{-\rho,b}}.
\end{align}
\end{prop}

This proposition is the main new technical result of the paper.

\begin{proof}
The idea of proof is similar to the one for (\ref{bilinear1}) in \cite{MP}, working in frequencies, and spliting in various domains. The main difference is that we will take advantage of the absence of derivative $\partial_x$ in the estimate (\ref{bilinear2}) compared to (\ref{bilinear1}) to work with a lower space regularity $-\rho<0$ instead of $s>\frac{1}{2}$.

The desired estimate (\ref{bilinear2}) is equivalent by duality to prove
\begin{align*}
\int_{\mathbb{R}^6} \f{w_0} \f{u_1} \f{v_2} \frac{\langle \vert (\xi_1,\eta_1)\vert \rangle^\rho \langle \vert (\xi_2,\eta_2)\vert \rangle^\rho}{\langle \vert (\xi_0,\eta_0)\vert \rangle^\rho} \frac{d(\tau_0,\xi_0, \mu_0)d(\tau_1, \xi_1,\mu_1)}{ \langle \tau_0-\omega_0 \rangle^{-b'} \langle \tau_1-\omega_1 \rangle^b \langle \tau_2-\omega_2 \rangle^b} \lesssim \|w_0 \|_{L^2} \|u_1 \|_{L^2} \|v_2 \|_{L^2},
\end{align*}
where $\xi_2:=\xi_0-\xi_1$, $\mu_2:= \mu_0-\mu_1$, $\tau_2:= \tau_0-\tau_1$, $\f{w_0}:=\f{w}(\tau_0,\xi_0, \mu_0)$, $\f{u_1}:=\f{u}(\tau_1,\xi_1, \mu_1)$, $\f{v_2}:=\f{v}(\tau_2,\xi_2, \mu_2)$, and $\omega_i:= \xi_i \left( \xi_i^2+\mu_i^2 \right)$.

We decompose along different frequencies: given integers $N_0, N_1, N_2$, let
\begin{align*}
I_{N_0,N_1,N_2}:= \int_{\mathbb{R}^6} \f{P_{N_0}w_0} \f{P_{N_1}u_1} \f{P_{N_2}v_2} \frac{\langle \vert (\xi_1,\eta_1)\vert \rangle^\rho \langle \vert (\xi_2,\eta_2)\vert \rangle^\rho}{\langle \vert (\xi_0,\eta_0)\vert \rangle^\rho} \frac{d(\tau_0,\xi_0, \mu_0)d(\tau_1, \xi_1,\mu_1)}{ \langle \tau_0-\omega_0 \rangle^{-b'} \langle \tau_1-\omega_1 \rangle^b\langle \tau_2-\omega_2 \rangle^b}.
\end{align*}
Analoguously, for the operators $Q_L$, given furthermore  integers $L_0, L_1, L_2$, denote
\begin{align*}
I_{N_0,N_1,N_2}^{L_0,L_1,L_2}:= \int_{\mathbb{R}^6} \f{(Q_{L_0}P_{N_0}w_0)} \f{(Q_{L_1}P_{N_1}u_1)} \f{(Q_{L_2}P_{N_2}v_2)} \frac{\langle \vert (\xi_1,\eta_1)\vert \rangle^\rho \langle \vert (\xi_2,\eta_2)\vert \rangle^\rho d(\tau_0,\xi_0, \mu_0)d(\tau_1, \xi_1,\mu_1)}{\langle \vert (\xi_0,\eta_0)\vert \rangle^\rho \langle \tau_0-\omega_0 \rangle^{-b'} \langle \tau_1-\omega_1 \rangle^b \langle \tau_2-\omega_2 \rangle^b}.
\end{align*}

We split the frequencies into five main domains.

\textbf{First domain.} $N_1 \le 2$, $N_2 \le 2$ and $N_0 \le 2$. 
We use the Plancherel equality and Hölder inequality:
\begin{align*}
\left\vert I_{N_0,N_1,N_2} \right\vert \lesssim \left\| \left(\frac{\f{u_1}}{\langle \tau_1- \omega_1 \rangle ^b} \right)^\vee \right\|_{L^4} \left\| \left( \frac{\f{v_2} }{\langle \tau_2- \omega_2 \rangle^b} \right)^\vee \right\|_{L^4} \| w_0 \|_{L^2},
\end{align*}
and we conclude by the adequate embedding (\ref{embedding}) : $X^{0,\frac{5}{12}^+} \hookrightarrow L^4$.

\textbf{Second domain.} $4 \le N_1$, $N_2 \le \frac{N_1}{4}$ so $ \frac{N_1}{2} \le N_0 \le 2N_1$. We use the upper bound of the localized frequencies (\ref{measure2}):
\begin{align*}
I_{N_0,N_1,N_2}^{L_0,L_1,L_2}
	& \lesssim \frac{N_1^\rho N_2^\rho}{N_0^\rho L_0^{-b'}L_1^bL_2^b} \left\| \left( P_{N_1} Q_{L_1}u_1 \right) \left( P_{N_2} Q_{L_2} v_2 \right) \right\|_{L^2} \|P_{N_0} Q_{L_0} w_0 \|_{L^2} \\
	& \lesssim \frac{N_1^\rho}{L_0^{-b'}L_1^b L_2^b} \frac{N_2^\frac{1}{2}}{N_1}L_1^\frac{1}{2}L_2^\frac{1}{2}\|P_{N_1}Q_{L_1} u_1 \|_{L^2} \| P_{N_2} Q_{L_2} v_2 \|_{L^2} \|P_{N_0} Q_{L_0} w_0 \|_{L^2}.
\end{align*}
For $\rho< \frac{1}{2}$, the sum over $L_0$, $L_1$, $L_2$, $N_1$, $N_2$ and $N_0$ gives:
\begin{align*}
\sum_{N_1 \le 2, N_2 \le \frac{N_1}{4},N_0}I_{N_0,N_1,N_2} \lesssim \|w_0 \|_{L^2} \|u_1 \|_{L^2} \| v_2 \|_{L^2}.
\end{align*}

\textbf{Third domain.} $4 \le N_2$, $N_1 \le \frac{N_2}{4}$ so $\frac{N_2}{2} \le N \le 2N_2$. By symmetry between $N_1$ and $N_2$, the third domain is dealt with as for the second domain.

\textbf{Fourth domain.} $4 \le N_1$, $N_0 \le \frac{N_1}{4}$ so $\frac{N_1}{2} \le N_2 \le 2N_1$. (or equivalently, $4 \le N_2$, $N_0 \le \frac{N_2}{4}$ so $\frac{N_2}{2} \le N_1 \le 2N_2$). We use the same decomposition as the second domain, an interpolation between (\ref{measure}) and (\ref{measure2}) by a coefficient $\theta \in (0,1)$, with $\widetilde{f}(x,y,z):=f(-x,-y,-z)$:
\begin{align*}
I_{N_0,N_1,N_2}^{L_0,L_1,L_2}
	& \lesssim \frac{N_1^\rho N_2^\rho}{N_0^\rho L_0^{-b'}L_1^bL_2^b} \left\| \left( \widetilde{P_{N_1} Q_{L_1} u_1}\right) P_{N_0} Q_{L_0} w_0 \right\|_{L^2} \|P_{N_2} Q_{L_2}v_2 \|_{L^2} \\
	& \lesssim \frac{N_1^\rho}{L_0^{-b'} L_1^b L_2^b}\frac{N_0^{\frac{1}{2}(1+\theta)}}{N_1^{1-\theta}}L_0^{\frac{1}{2}(1-\theta)} L_1^{\frac{1}{2}} \| P_{N_0} Q_{L_0}w_0 \|_{L^2} \|P_{N_1} Q_{L_1}u_1\|_{L^2} \|P_{N_2} Q_{L_2}v_2 \|_{L^2}\\
	& \lesssim \frac{N_1^{-\frac{1}{2}+ \frac{3}{2}\theta+\rho}}{L_0^{-b'+\frac{\theta}{2} -\frac{1}{2}}L_1^{b-\frac{1}{2}} L_2^b}\| P_{N_0} Q_{L_0}w_0 \|_{L^2} \|P_{N_1} Q_{L_1}u_1\|_{L^2} \|P_{N_2} Q_{L_2}v_2 \|_{L^2}.
\end{align*}
By choosing $\theta$ such that $ 4\delta <\theta < \frac{1}{3} \left( 1-2 \rho \right)$ (this is one of the points giving the bounds on $\delta$ and $\rho$), the sum over $L_0$, $L_1$, $L_2$, $N_0$, $N_1$ and $N_2$ is bounded by  $\|w_0 \|_{L^2} \|u_1\|_{L^2} \|v_2 \|_{L^2}$.

\textbf{Fifth domain.} $4 \le N_1$, $4 \le N_2$, $N_0 \ge \frac{N_1}{2}$ and $N_0 \ge \frac{N_2}{2}$ (so $\frac{N_2}{2}\le N_1 \le 2N_2$, $\frac{N_1}{2} \le N_0 \le 2 N_1$ and $\frac{N_2}{2} \le N_0 \le 2N_2$).

We divide this domain depending on the values of $\xi_i$ and $\mu_i$, with a coefficient $\alpha$ to define later:
\begin{align*}
\begin{array}{l}
\mathcal{D}_1:=\left\{ \left(\tau_0, \xi_0, \mu_0, \tau_1, \xi_1, \mu_1 \right) ; (1-\alpha)^\frac{1}{2} \sqrt{3}\vert \xi_i \vert \le \vert \mu_i \vert \le (1-\alpha)^{-\frac{1}{2}} \sqrt{3} \vert \xi_i \vert, i=1,2 \right\}, \\
 \mathcal{D}_2:=\left\{ \left(\tau_0, \xi_0, \mu_0, \tau_1, \xi_1, \mu_1 \right) ; (1-\alpha)^\frac{1}{2} \sqrt{3}\vert \xi_i \vert \le \vert \mu_i \vert \le (1-\alpha)^{-\frac{1}{2}} \sqrt{3} \vert \xi_i \vert, i=0,1 \right\}, \\
 \mathcal{D}_3:=\left\{ \left(\tau_0, \xi_0, \mu_0, \tau_1, \xi_1, \mu_1 \right) ; (1-\alpha)^\frac{1}{2} \sqrt{3}\vert \xi_i \vert \le \vert \mu_i \vert \le (1-\alpha)^{-\frac{1}{2}} \sqrt{3} \vert \xi_i \vert, i=0,2 \right\}, \\
\mathcal{D}_4:=\mathbb{R}^6_{\backslash \mathcal{D}_1 \cup\mathcal{D}_2 \cup\mathcal{D}_3}.
\end{array}
\end{align*}
First, we work on $\mathcal{D}_1$, in a subregion where $\xi_1 \xi_2>0$ and $\mu_1\mu_2>0$. We claim that:
\begin{align*}
N_1^3 \lesssim \max \left\{ \vert \tau_0-\omega_0 \vert, \vert \tau_1-\omega_1 \vert, \vert \tau_2-\omega_2 \vert \right\}.
\end{align*}
In fact, if $\vert \tau_1-\omega_1 \vert + \vert \tau_2- \omega_2 \vert \lesssim N_1^3$, then:
\begin{align*}
\tau_0 - \omega_0 =\tau_1-\omega_1 +\tau_2 -\omega_2 -  \xi_1 \left( \xi_2^2 + 2 \xi_1 \xi_2 + \mu_2^2 +2 \mu_1 \mu_2 \right) -  \xi_2 \left( \xi_1^2 + 2 \xi_1 \xi_2 + \mu_1^2 +2 \mu_1 \mu_2 \right) \simeq \pm N_1^3,
\end{align*}
which proves the claim if $\xi_1 \xi_2 >0$ and $\mu_1 \mu_2 >0$. The other cases are similar. We can thus argue as in the first domain, and obtain that for $ \rho \le \frac{3}{2}-6\delta$,
\begin{align*}
\sum_{N_0,N_1,N_2} I_{N_0,N_1,N_2} 
	& \lesssim \sum_{N_0,N_1,N_2}N_1^{\rho+3b'} \left\| \left( \frac{\f{P_{N_1} u_1}}{\langle \tau_1-\omega_1 \rangle^b} \right)^\vee \right\|_{X^{0, \frac{5}{6}^+}} \left\| \left( \frac{\f{P_{N_2} u_2}}{\langle\tau_2-\omega_2 \rangle^b} \right)^\vee \right\|_{X^{0, \frac{5}{6}^+}} \| P_{N_0} w_0 \|_{L^2} \\
	&  \lesssim \| w_0 \|_{L^2} \|u_1 \|_{L^2} \|v_2 \|_{L^2}.
\end{align*} 
Next, the subregions $\xi_1 \xi_2>0$ and $\mu_1 \mu_2<0$, or $\xi_1 \xi_2 <0$ and $\mu_1 \mu_2 >0$, we use the dyadic decomposition along the flow with the operators $Q_L$ and by a Cauchy-Schwarz inequality:
\begin{align*}
I_{N_0,N_1,N_2}^{L_0,L_1,L_2} \lesssim \frac{N_1^\rho }{L_0^{-b'} L_1^b L_2^b} \left\| \left( P_{N_1} Q_{L_1} u_1 \right) \left( P_{N_2} Q_{L_2} v_2 \right) \right\|_{L^2} \| P_{N_0} Q_{L_0} w_0 \| _{L^2}.
\end{align*}

To deal with the first $L^2$-norm, we need to bound the interaction of the high/high frequencies. This is the purpose of the following lemma.

\begin{lemm}\label{lemma11}
Let $N_1$, $N_2$, $L_1$ and $L_2$ be four integers involved in the operators $P_N$ and $Q_L$, and two functions $u_1$ and $v_2$. 
In the case $\frac{N_1}{2} \le N_2 \le 2N_1$, define the subsets $S_1$ and $S_2$ of $\mathbb{R}^2_{\xi_1, \mu_1} \times \mathbb{R}^2_{\xi_2, \mu_2}$ by 
\begin{align*}
S_1(\xi_1,\xi_2, \mu_1, \mu_2) & :=\left\{ \xi_1 \xi_2 >0 ,\mu_1 \mu_2 <0 \right\} \cup \left\{ \xi_1 \xi_2 <0 , \mu_1 \mu_2 >0 \right\}, \\
S_2(\xi_1,\xi_2, \mu_1, \mu_2) & :=\left\{ \xi_1 \xi_2 <0 ,\mu_1 \mu_2 <0 \right\},
\end{align*}
and the two Fourier multipliers $(J_i)_{1 \le i \le 2}$, which take into account only some interaction of frequencies, by:
\begin{align*}
\f{J_i(u_1,v_2)}(\tau, \xi, \mu):= \int_{\mathbb{R}^3} \mathbf{1}_{S_i (\xi_1,\xi-\xi_1, \mu_1, \mu-\mu_1)} \f{u_1}(\tau_1, \xi_1, \mu_1) \f{v_2}(\tau-\tau_1, \xi-\xi_1,\mu-\mu_1)d\tau_1d\xi_1 d\mu_1.
\end{align*}
Then, for $\alpha$ small enough, one has
\begin{align}\label{measure3}
\left\| J_i \left( P_{N_1} Q_{L_1} u_1, P_{N_2} Q_{L_2} v_2 \right) \right\|_{L^2} \lesssim N_1^{-\frac{1}{2}} (L_1L_2)^\frac{1}{2} \left\| P_{N_1} Q_{L_1} u_1 \right\|_{L^2} \left\|P_{N_2} Q_{L^2} v_2 \right\|_{L^2}.
\end{align}
\end{lemm}
Observe in (\ref{measure3}) the $-\frac{1}{2}$ gain for high/high interaction in $N_1^{-1/2}$; the smallness requirement on $\alpha$ comes from this lemma. Let us postpone its proof to the appendix, and conclude now the proof of Proposition \ref{bilinear2prop}.
We thus use inequality (\ref{measure3}) with $J_1$ to obtain:
\begin{align*}
I_{N_0,N_1,N_2}^{L_0,L_1,L_2} \lesssim \frac{1 }{N_1^{\frac{1}{2}-\rho}L_0^{-b'} L_1^{b-\frac{1}{2}} L_2^{b-\frac{1}{2}}} \left\|  P_{N_1} Q_{L_1} u_1 \right\|_{L^2} \left\| P_{N_2} Q_{L_2} v_2 \right\|_{L^2} \| P_{N_0} Q_{L_0} w_0 \| _{L^2},
\end{align*}
and conclude summing over $N_0$, $N_1$, $N_2$, $L_0$, $L_1$ and $L_2$. In the subregion $\xi_1 \xi_2 <0$, $\mu_1 \mu_2<0$, we use the same process with $J_2$ in (\ref{measure3}).

In the regions $\mathcal{D}_2$ and $\mathcal{D}_3$, the estimates are similar: it suffices to establish a change of variable which brings us exactly to the case of $\mathcal{D}_1$. Observe that there is no influence from neither $b$ nor $b'$.

Finally, in $\mathcal{D}_4$, we use the Fourier multiplier operator $K(\xi,\mu):=\vert 3 \xi^2 -\mu^2 \vert \gtrsim \vert \langle(\xi,\mu) \rangle \vert$, and the estimate (\ref{restriction_theorem}), so that for $\rho < \frac{1}{2}$:
\begin{align*}
I_{N_0, N_1,N_2} 
	& \lesssim N_1^{\rho-\frac{1}{2}} \left\| K(D)^\frac{1}{8} \left( \frac{\f{P_{N_1} u_1}}{\langle \tau_1-\omega_1 \rangle^b} \right)^\vee \right\|_{L^4} \left\| K(D)^\frac{1}{8} \left( \frac{\f{P_{N_2} v_2}}{\langle \tau_2-\omega_2 \rangle^b} \right)^\vee \right\|_{L^4} \| P_{N_0} w_0 \|_{L^2} \\
	& \lesssim \| P_{N_0} w_0 \|_{L^2} \|P_{N_1} u_1 \|_{L^2} \| P_{N_2} v_2 \|_{L^2}. \qedhere
\end{align*}
\end{proof}

\section{Growth of Sobolev norms}

We start this section with a local well posedness result where we carefully track the dependency of the existence time: it is crucial that it only depends on $\| u_0 \|_{H^1}$. 

In previous results set in the context of Bourgain spaces, the proofs were made with dependency on $\| u_0 \|_{H^s}$: for instance, in Kenig-Ponce-Vega \cite{KPV96} (where they assume $\| u_0 \|_{H^s} \le 1$ and then use a scaling argument) or Molinet Pilod \cite{MP}. This is very good for low $s<1$, but does not quite give a suitable result in our perspective of large $s >1$. This is why we provide a full proof below, using in particular the tame bilinear estimate \eqref{bilinear3} (which is only relevant for $s >1$).

\bigskip

Throughout the remainder of this section we fix
\[ s>1, \quad \text{and} \quad \delta \in \left( 0, \frac{1}{12} \right). \]
(In particular, the constant $C_0$ and the function $T$ below depend on $s$ and $\delta$).

\begin{prop} \label{theo_point_fixe}
There exists a universal constant $C_0>0$ such that if we define
\begin{align} \label{T}
\forall A >0, \quad T(A):= \frac{1}{(8C_0^2 A)^{\frac{1}{\delta}}},
\end{align}
there exists a unique solution $u \in \mathcal{C}([-T(\| u_0 \|_{H^1}), T(\| u_0 \|_{H^1})],H^s)$ of (\ref{ZK}) with initial condition $u(0)=u_0$. Furthermore, $u$ satisfies the estimates :
\begin{align}\label{Hs_Xsb}
\| u \|_{\mathcal{C}([-T(\| u_0 \|_{H^1}), T(\| u_0 \|_{H^1})],H^s)} + \| \phi_{T(\| u_0 \|_{H^1})}u\|_{X^{s,\frac{1}{2}+\delta}} \le C_0 \|u_0\|_{H^s} \text{ and }\| \phi_{T(\| u_0 \|_{H^1})}u\|_{X^{1,\frac{1}{2}+\delta}} \le C_0 \|u_0\|_{H^1}.
\end{align}
\end{prop}

\begin{proof}
We consider the initial condition as $u_0$, with a finite norm in $H^s$, $s \ge 2$. We work on
\begin{align*}
B:=\left\{ w \in X^{s,b}; \| w \|_{X^{s,b}} \le 2C_0 \|u_0 \|_{H^s} \text{ and } \|w \|_{X^{1,b}} \le 2C_0 \| u_0 \|_{H^1}\right\},
\end{align*}
with the distance $d(v,w) := \| v-w \|_{X^{1,b}}$. $(B,d)$ is a complete metric space.

Let $T>0$ to be chosen later. We want to apply the fixed point theorem to the function $\Psi$ defined in (\ref{Psi_1}) on the space $B$. Let $\delta$ be given as in the bilinear estimates (\ref{bilinear1}) and (\ref{bilinear3}), $b:=\frac{1}{2}+ \delta$ and $b' :=2b-\frac{3}{2}=-\frac{1}{2}+2\delta$. Using lemma \ref{lemma5} and proposition \ref{propo8} :
\begin{align*}
\| \Psi(w)\|_{X^{s,b}} 
	& \le \| \phi_1(t) W(t) u_0 \|_{X^{s,b}} + \left\| \phi_T(t) \int_0^t W(t-t') \phi_T(t')^2\partial_x \left( w^2  \right)(t')dt'\right\|_{X^{s,b}}\\
	& \le C_0  \left( \|u_0\|_{H^s} + T^{1-b+b'} \left\| \partial_x \left(\left( \phi_T(t)w\right)^2 \right) \right\|_{X^{s,b'}} \right) \\
	& \le C_0 \|u_0\|_{H^s} \left( 1+ 4C_0^2 T^\delta \| u_0 \|_{H^1}\right),
\end{align*}
and similarly, (increasing the constant $C_0$ if necessary),
\begin{align*}
\| \Psi(w)\|_{X^{1,b}} \le C_0  \|u_0\|_{H^1} \left(1+ 4 T^\delta C_0^2\| u_0 \|_{H^1}\right).
\end{align*}
Furthermore, the function $\Psi$ is a contraction on $(B,d)$:
\begin{align*}
\| \Psi(\omega)-\Psi(\widetilde{\omega}) \|_{X^{1,b}} \le C_0 T^\delta \| \omega-\widetilde{\omega} \|_{X^{1,b}} \| \omega + \widetilde{\omega} \|_{X^{1,b}} \le 4 C_0^2 T^\delta \|u_0\|_{H^1}\| \omega-\widetilde{\omega} \|_{X^{1,b}}.
\end{align*}

Hence, by choosing $T>0$ such that $T^\delta =\frac{1}{8C_0^2 \|u_0 \|_{H^1}}$, we can apply the fixed point theorem. Hence we obtain 
\begin{align*}
\| \phi_T u \|_{X^{s,b}} \le C_0 \|u_0 \|_{H^s}.
\end{align*}
The $\| u \|_{\mathcal{C}([-T(\| u_0 \|_{H^1}), T(\| u_0 \|_{H^1})],H^s)}$ bound is immediate from \eqref{embedding2}.

Finally, the proof of continuity of the unique solution $u \in \mathcal{C}([0,T],H^s)$ is inspired by Kenig-Ponce-Vega \cite{KPV96}. We consider continuity at $0$, the other points are similar. Fix $t > 0$. Then the Duhamel formula can be rewritten
\[ u(t) = W(t) u_0 + \int_0^t W(t-t')  \partial_x \left( (\phi_t(t')u(t'))^2 \right) dt'. \]
We use successively embedding (\ref{embedding2}), estimate (\ref{autre_inclusion}) and estimate (\ref{bilinear1}):
\begin{align*}
\left\| u(t)-u(0) \right\|_{H^s} 
	& \le \left\| W(t) u(0)-u(0) \right\|_{H^s} + \left\| \int_0^t W(t-t')  \partial_x \left( (\phi_t(t')u(t'))^2 \right) dt' \right\|_{H^s} \\
	& \le \left\| W(t) u(0)-u(0) \right\|_{H^s} + C(b) \left\| \int_0^t W(t-t')  \partial_x \left( (\phi_t(t')u(t'))^2 \right) dt' \right\|_{X^{s,b}} \\
	& \le \left\| W(t) u(0)-u(0) \right\|_{H^s} + C(b) t^{1-b+b'} \left\|   \partial_x \left( (\phi_t u)^2 \right) \right\|_{X^{s,b'}} \\
	& \le \left\| W(t) u(0)-u(0) \right\|_{H^s} + C(b,s) t^{\delta} \left\|  \phi_t u  \right\|_{X^{s,b}}^2,
\end{align*}
which tends to $0$ as $t \to 0$. This proves continuity at $0$ in $H^s$.
\end{proof}

We will now prove Theorem \ref{th1}, and for this assume that $s>1$ is an integer.

\begin{proof}[Proof of Theorem \ref{th1}]
We denote $u \in \mathcal C(I, H^s)$ to be the maximal $H^s$ development of $u_0$ under the \eqref{ZK} flow. Denote 
\[ A := \sup_{t \in I} \| u(t) \|_{H^1} \le \sqrt{C(E(u_0) + M(u_0)^2)}. \]
(See \eqref{est:H1_ME}). The equation (\ref{ZK}) is reversible, it suffices to prove \eqref{control_poly} for positive times $t \ge 0$. 

Proposition \ref{theo_point_fixe} yields that for all $t_0 \in I$, $[t_0 - T(A), t_0+T(A)] \subset I$, and
\[ \forall t \in [t_0 - T(A), t_0+T(A)], \quad \| u(t) \|_{H^s} \le C_0 \| u(t_0) \|_{H^s}. \]
This shows by iteration that $I = \mathbb R$, and that for all $k \in \mathbb N$,
\[ \forall t \in [kT(A), (k+1) T(A)], \quad \| u(t) \|_{H^s} \le C_0^{k+1} \| u_0 \|_{H^s}. \]
From this, we infer the exponential bound $\| u(t) \|_{H^s} \le \exp(C t)$ for some large constant $C$ and all $t \ge 0$.

To improve this to a polynomial bound, we will follow the idea of  Staffilani \cite{Sta} to obtain more refined inequality, namely
\begin{align}\label{ineq}
\forall t \in [ 0, T(\| u_0 \|_{H^1})], \quad \| u(t) \|_{H^s}^2-\| u(0)\|_{H^s}^2 \le C(\| u_0 \|_{H^1}) (1+\|u_0 \|_{H^s}^{2-\epsilon}),
\end{align}
for some $\epsilon$ close to $0$ and some function $C: [0,+\infty) \to [0,+\infty)$ . 

This will imply that for all $t_0 \in \mathbb R$, and for some uniform constant $C_1 \ge 1$ (depending on $A$ and $C$)
\[ \forall t \in [t_0 - T(A), t_0 + T(A)], \quad  \| u(t) \|_{H^s}^2 \le \| u(t_0)\|_{H^s}^2 + C_1 (1+ \| u(t_0) \|_{H^s}^{2-\epsilon}). \]
Therefore, denoting $u_k = \sup_{t \in [kT(A),(k+1) T(A)]} \| u(t) \|_{H^s}$, then (with $t_0 = kT$)
\[ \forall k \in \mathbb N, \quad u_{k+1} \le \left(u_k^2 + C_1 (1+ u_k^{2- \epsilon}) \right)^{1/2} \le u_k + C_1 (1+ u_k^{1- \epsilon}). \]
As observed in \cite{Sta}, an iteration argument then shows that $u_k$ has polynomial growth: this is the content of the following lemma.
 
 \begin{lemm} \label{lem13}
 Let $(u_k)_{k \in \mathbb N}$ be a sequence of non negative numbers, $K_1>0$, $\epsilon \in (0,1)$ such that
 \[ \forall k \in \mathbb N, \quad u_{k+1} \le u_k + K_1(1+ u_k^{1-\epsilon}). \]
 Then for any $d > \frac{1}{\epsilon}$, there exist a constant $K_2=K_2(K_1,d)$ such that
 \[ \forall k \in \mathbb N, \quad u_{k+1} \le K_2 (1+k)^{d} (1+u_0). \]
 \end{lemm}
 
 \begin{proof}
Observe that due to convexity of $x \mapsto |x|^d$, for all $k >0$,
 \begin{equation} \label{est:conv}
 (2+k)^d = (1+k)^d \left( 1 + \frac{1}{1+k} \right)^d \ge (1+k)^d + d(1+k)^{d-1}.
 \end{equation}
 Also, as $1+u_{k+1} \le 1+u_k + K_1 (1+u_k^{1-\epsilon}) \le (2K_1+1)(1+u_k)$, we infer that $u_k \le (2K_1+1)^{k} (1+u_0)$ for all $k \in \mathbb N$.
 
 Let $N \in \mathbb N$ such that $N \ge \left( \frac{K_1}{d} \right)^{1/(d\epsilon-1)}$, so that for $k \ge N$ (as $1-d\epsilon <0$)
 \[ d- K_1 (1+k)^{1-d\epsilon} \ge 1. \]
 Let $K_2 := \max \left((2K_1+1)^N, \frac{K_1}{N^{d-1}} \right)$, so that 
 \begin{enumerate}
 \item for $k \le N$, $u_k \le K_2 (1+u_0) \le K_2(1+k)^d (1+u_0)$ and
 \item for $k \ge N$,
 \[ K_2 (1+k)^{d-1} \left( d- K_1 (1+k)^{1-d\epsilon} \right) \ge K_2 (1+k)^{d-1} \ge K_2 N^{d-1} \ge K_1, \]
 and so
 \begin{equation} \label{def:K2_rec}
 K_2 d(1+k)^{d-1} \ge K_1 + K_1 K_2 (1+k)^{d-1+ 1-d\epsilon} \ge K_1 + K_1 K_2 (1+k)^{d(1-\epsilon)}.
 \end{equation}
 \end{enumerate}
We prove now the statement by induction on $k$. For $k =0, \dots, N$, the statement hold due to 1. On the other side, if for some $k \ge N$, there hold $u_k \le  K_2(1+k)^d$ , then using  \eqref{est:conv} and \eqref{def:K2_rec},
 \begin{align*}
 u_{k+1} \le  K_2 (1+k)^{d} + K_1 (1 +  K_2 (1+k)^{d(1-\epsilon)} ) \le K_2 (1+k)^{d} + K_2 d(1+k)^{d-1} \le K_2 (2+k)^d.
 \end{align*}
 This proves the induction step for $k \ge N$, and the proof is complete.
 \end{proof} 
 
 From the above lemma, $u_k \lesssim (1+k)^{\frac{1}{\epsilon}^+} (1+ \| u_0 \|_{H^s})$ (with implicit constant depending on $A$), and from there the polynomial bound \eqref{control_poly} follows immediately with $\beta = \displaystyle \frac{1}{\epsilon}^+$ . We can now focus on proving \eqref{ineq}.

We recall that $T$ has been defined in the previous proposition, and $s\ge 2$ is an integer. Let $t\in [0,T]$, $b$ and $b'$ as in the proposition \ref{bilinear2prop}, satisfying $-b' \le b$. We control the evolution of the $H^s$-norm:
\begin{align*}
\| u(t)\|_{H^s}^2-\|u(0)\|_{H^s}^2  & = \int_0^t \int \langle \vert (\xi,\mu)\vert \rangle^{2s}\frac{d}{dt}\vert \fxy{u}(t') \vert^2 d\xi d\mu dt' \\
	& = \int_0^t \int \langle \vert (\xi,\mu)\vert \rangle^{2s}  \phi_T(t')^3 \left(\fxy{u} \fxy{\partial_x(u^2)} \right)(t') d\xi d\mu dt' \\
	& \simeq \sum_{\vert \bi \vert \le s} \int_0^t \int \fxy{D^\bi \phi_T u} \fxy{ D^\bi \partial_x(\phi_T^2 u^2) }  (t')d\xi d\mu dt'=I_0 + I_{1 \cdots s-1} + I_{s},
\end{align*}
where $I_0$ counts the term with none derivative, $I_s$ the terms with $s$ derivatives, and $I_{1 \cdots s-1}$ the others.

Let first deal with $I_0$, thus the case $\vert \bi \vert=0$. We obtain:
\begin{align*}
I_0= \int _0^t \int \phi_T u \partial_x \left( \phi_T^2 u^2 \right)(t') dxdydt =0.
\end{align*}

Second, let us find an upper bound for $I_{1 \cdots s-1}$, the term for which the sum of the derivatives does not reach $s$.  By a Cauchy-Schwarz inequality, recalling the function $F$ to follow the flow in (\ref{F}), we find the flow of our solution, up to a power $b'$:
\begin{align*}
I_{1 \cdots s-1} & = \sum_{0< \vert \bi \vert < s} \int_0^t \int \fxy{D^\bi \phi_T u} \fxy{ D^\bi \partial_x(\phi_T^2 u^2) }  (t')d\xi d\mu dt' \\
& \le \sum_{0< \vert \bi \vert < s} \int_0^t \left\| F(D)^{-b'} D^\bi (\phi_T u)(t') \right\|_{L^2} \left\| F(D)^{b'} \partial_x D^\bi \left( \phi_T^2 u^2\right)(t')\right\|_{L^2} dt'.
\end{align*}

We can apply estimate (\ref{bilinear1}) to the term on the right because $\vert \bi \vert \ge 1$. We use the estimate (\ref{bilinear1}), a Gagliardo-Nirenberg inequality, and twice estimate (\ref{differentb}):
\begin{align}
\MoveEqLeft
	I_{1 \cdots s-1} \le \sum_{0<\vert \bi \vert < s}\| D^\bi \left(\phi_T u\right) \|_{X^{0,-b'}}  \sum_{(0,0) \le \bj \le \bi } \| \phi_T u \|_{X^{\vert \bj \vert,b}} \|  \phi_T u \|_{X^{\vert\bi-\bj \vert,b}} \nonumber \\
	& \lesssim  \sum_{0< \vert \bi \vert < s} \|\phi_T u\|_{X^{s,-b'}}^{\vert \bi \vert /s} \| \phi_T u\|_{X^{0,-b'}}^{1-\vert \bi \vert/s} \sum_{(0,0) < \bj < \bi } \|\phi_T u\|_{X^{s,b}}^{\vert \bi \vert /s} \max \left( \| \phi_T u\|_{X^{0,b}}^{1-\vert \bj \vert /s} , \| \phi_T u\|_{X^{0,b}} \right) \nonumber \\
	& \le C \max(1,\| \phi_T u\|_{X^{0,b}}^2 ) \max \left( \| \phi_T u \|_{X^{s,b}}^{\frac{2}{s}} , \|\phi_T u\|_{X^{s,b}}^{2(1-\frac{1}{s})}\right). \nonumber
\end{align}

Third, the terms $I_s$ with $\vert \bi \vert=s$. We separate $I_s=I_s^1+I_s^2$, where $I_s^2$ contains the two terms where all the derivatives are distributed on the same term, and $I_s^1$ all the others. Let $0<\rho <\frac{1}{2}$, and $\alpha_\bj:= \frac{\vert \bj \vert -1+\rho}{s-1}$. We deal first with $I_s^1$, the terms where all the derivatives are not on the same term. By using the estimate (\ref{bilinear1}):
\begin{align*}
I_s^1: & = \sum_{\vert \bi \vert =s} \sum_{ (0,0) < \bj < \bi} \iint \left\vert D^\bi \left( \phi_T u \right) \partial_x \left( D^\bj \left( \phi_T u \right) D^{\bi-\bj}\left( \phi_T u \right)\right) \right\vert (t')dxdydt' \\
	& \lesssim \sum_{\vert \bi \vert =s} \sum_{(0,0) < \bj < \bi} \| D^\bi \left(\phi_T u\right) \|_{X^{-\rho,-b'}}  \| D^\bj \left( \phi_T u \right) \|_{X^{\rho,b}} \| D^{\bi-\bj}\left( \phi_T u \right) \|_{X^{\rho,b}} \\
	& \lesssim \| \phi_T u \|_{X^{s-\rho,-b'}} \sum_{(0,0) < \bj < \bi} \| \phi_T u \|_{X^{s,b}} ^{\alpha_\bj} \| \phi_T u \|_{X^{1,b}} ^{1-\alpha_\bj} \| \phi_T u \|_{X^{s,b}} ^{\alpha_{\bi-\bj}} \| \phi_T u \|_{X^{1,b}} ^{1-\alpha_{\bi-\bj}}  \\
	& \lesssim  \| \phi_T u \|_{X^{s-\rho,-b'}} \| \phi_T u \|_{X^{s,b}} ^{1+\frac{2\rho-1}{s-1}} \max( 1,\|\phi_T u \|_{X^{1,b}} ^{\frac{s-2\rho}{s-1}}) \\
	& \le C \max(1, \| \phi_T u \|_{X^{1,b}}^2) \| \phi_T u \|_{X^{s,b}}^{2-\frac{1-\rho}{s-1}},
\end{align*}
where the last line comes from a final interpolation. One can notice that the last estimate is sharp, and that $\frac{1}{2(s-1)}<\frac{1-\rho}{s-1}$.

We can not directly use the bilinear estimate (\ref{bilinear1}) for the term $I_s^2$ where all the derivatives are distributed on one term, due to the lower limit of $s>\frac{1}{2}$. This term is dealt with using estimate (\ref{bilinear2}) as it authorizes a negative index in $X^{-\rho,b}$. Let $\rho>0$ be given by Proposition \ref{bilinear2prop}. We recall the definition of $S$ in (\ref{S}) and the Gagliardo-Nirenberg inequality $\| D^\rho v \|_{L_{x,y}^2} \lesssim \| D^{s-1}v \|_{L_{x,y}^2}^\frac{\rho}{s-1} \| v \|_{L_{x,y}^2}^{1-\frac{\rho}{s-1}}$. Using an integration by parts when needed, we can bound
\begin{align*}
I_s^2 & := \sum_{\vert \bi \vert =s} \iint \left( D^\bi \left( \phi_T u \right) \partial_x \left( \phi_T u D^\bi(\phi_T u)\right) \right) (t')dxdydt' \\
	& \le \sum_{\vert \bi \vert =s} \iint \left\vert \partial_x(\phi_T u) \right\vert \left\vert D^\bi (\phi_T u) \right\vert^2dxdydt'  \le \| \partial_x \phi_T u \|_{X^{\rho,-b'}} \left\| \left( D^\bi (\phi_T u )\right)^2 \right\|_{X^{-\rho,b'}} \\ 
	& \lesssim \| S(D)^\rho \partial_x \phi_T u \|_{X^{0,b}} \left\| S(D)^{s-\rho}(\phi_T u )\right\|_{X^{0,b}}^2  \lesssim \left( \| \phi_T u \|_{X^{1,b}}^{1-\frac{\rho}{s-1}} \| \phi_T u \|_{X^{s,b}}^\frac{\rho}{s-1} \right) \| \phi_T u \|_{X^{s,b}}^{2 \frac{s-\rho-1}{s-1}} \| \phi_T u \|_{X^{1,b}}^{2\frac{\rho}{s-1}} \\
	& \le C \| \phi_T u \|_{X^{1,b}}^{1+\frac{\rho}{s-1}} \| \phi_T u \|_{X^{s,b}}^{2-\frac{\rho}{s-1}}.
\end{align*}

We thus notice that this time, $0< \frac{\rho}{s-1} <\frac{1}{2(s-1)}$.

Summing up the three previous inequalities, we obtain :
\begin{align*}
\| u (t) \|_{H^s}^2 - \| u(0) \|_{H^s}^2 \le C \max( 1, \| \phi_T u \|_{X^{1,b}}^2) \max \left( \| \phi_T u \|_{X^{s,b}}^{\frac{2}{s}}, \| \phi_T u \|_{X^{s,b,}}^{2-\frac{\rho}{s-1}} \right)
\end{align*}
and by the embeddings (\ref{Hs_Xsb}):
\begin{align*}
\| u(t)\|_{H^s} \le \|u(0)\|_{H^s} + C \max (1, C_0 \| u(0) \|_{H^1}) C_0^{1-\frac{\rho}{2(s-1)}} \max \left(\| u(0)\|_{H^s}^{\frac{1}{s}} , \| u(0) \|_{H^s}^{1-\frac{\rho}{2(s-1)}}\right).
\end{align*}
This concludes the proof of the estimate (\ref{ineq}), by choosing $\epsilon= \frac{\rho}{s-1}$ and $\rho \in \left( 0,\frac{1}{2} \right)$.
\end{proof}

\section{Appendix : Proof of lemma \ref{lemma11} }

This lemma is specific to (total) dimension 3. Its proof is almost similar to that of the estimate 5.2 \cite[Proposition 3.6]{MP}, which is slightly different from \eqref{measure3}. We recall the assumption on the coefficients: $N_1$, $N_2$, $L_1$ and $L_2$ be four integers, two functions $u_1$ and $v_2$, and $\frac{N_1}{2} \le N_2 \le 2N_1$.
We define the subsets:
\begin{align*}
S_1(\xi_1,\xi_2, \mu_1, \mu_2) & :=\left\{ \xi_1 \xi_2 >0 ,\mu_1 \mu_2 <0 \right\} \cup \left\{ \xi_1 \xi_2 <0 , \mu_1 \mu_2 >0 \right\}, \\
S_2(\xi_1,\xi_2, \mu_1, \mu_2) & :=\left\{ \xi_1 \xi_2 <0 ,\mu_1 \mu_2 <0 \right\},
\end{align*}
and the two Fourier multipliers $(J_i)_{1 \le i \le 2}$, by:
\begin{align*}
\f{J_i(u_1,v_2)}(\tau, \xi, \mu):= \int_{\mathbb{R}^3} \mathbf{1}_{S_i (\xi_1,\xi-\xi_1, \mu_1, \mu-\mu_1)} \f{u_1}(\tau_1, \xi_1, \mu_1) \f{v_2}(\tau-\tau_1, \xi-\xi_1,\mu-\mu_1)d\tau_1d\xi_1 d\mu_1.
\end{align*}
We want to prove that:
\begin{align}\label{measure4}
\left\| J_i \left( P_{N_1} Q_{L_1} u_1, P_{N_2} Q_{L_2} v_2 \right) \right\|_{L^2} \lesssim N_1^{-\frac{1}{2}} (L_1L_2)^\frac{1}{2} \left\| P_{N_1} Q_{L_1} u_1 \right\|_{L^2} \left\|P_{N_2} Q_{L^2} v_2 \right\|_{L^2}.
\end{align}

We consider $J_1$ first, and explain at the end how $J_2$ will be different. First, the convolution of the functions brings us to define a set on which the functions are not null:
\begin{align*}
A_{(\xi, \mu,\tau)}:= \left\{ 
\begin{array}{l}
(\xi_1,\mu_1,\tau_1) \in \mathbb{R}^3; (\xi_1, \xi-\xi_1, \mu_1, \mu-\mu_1) \in S_1, \\
 \vert (\xi_1, \mu_1) \vert \simeq N_1, \vert (\xi-\xi_1,\mu- \mu_1) \vert \simeq N_2 ,\\
 \vert \tau_1- \omega(\xi_1, \mu_1) \vert \simeq L_1 , \vert \tau-\tau_1 - \omega(\xi-\xi_1, \mu- \mu_1) \vert \simeq L_2
\end{array}  \right\}.
\end{align*}  

We thus get, by Minkowski inequality, and Cauchy-Schwarz inequality:
\begin{align*}
\MoveEqLeft
\|J_1(P_{N_1} Q_{L_1} u_1, P_{N_2} Q_{L_2} v_2) \|_{L^2} \\
	& \le \int \left( \int \mathbf{1}_{ A_{(\xi,\mu,\tau)}}(\xi_1,\mu_1,\tau_1) \f{P_{N_1} Q_{L_1} u_1}(\xi_1,\mu_1,\tau_1)^2 \f{P_{N_2} Q_{L_2} v_2}(\xi-\xi_1,\mu-\mu_1,\tau- \tau_1)^2 d\xi d\mu d\tau \right)^\frac{1}{2}d\xi_1 d\mu_1 d\tau_1 \\
	& \le \| P_{N_1} Q_{L_1}u_1 \|_{L^2} \left( \int \int \mathbf{1}_{ A_{(\xi,\mu,\tau)}}(\xi_1,\mu_1,\tau_1) \vert P_{N_2} Q_{L_2} v_2 (\xi-\xi_1, \mu-\mu_1, \tau-\tau_1) \vert^2 d\xi d\mu d\tau d\xi_1 d\mu_1 d\tau_1 \right)^\frac{1}{2} \\
	& \le \sup_{\xi,\mu,\tau} \vert A_{(\xi,\mu,\tau)} \vert^\frac{1}{2} \| P_{N_1}Q_{L_1} u_1 \|_{L^2} \| P_{N_2} Q_{L_2} v_2 \|_{L^2}.
\end{align*}

To prove (\ref{measure4}), it suffices to bound:
\begin{align} \label{measure5}
\sup_{\xi,\mu, \tau} \vert A_{(\xi,\mu,\tau)} \vert \lesssim \frac{L_1^\frac{1}{2} L_2^\frac{1}{2}}{N_1^\frac{1}{2}}.
\end{align}

To do so, define the resonance function $\H(\xi_1,\xi_2,\mu_1,\mu_2):= \omega( \xi_1+\xi_2, \mu_1 + \mu_2) - \omega(\xi_1, \mu_1)-\omega(\xi_2,\mu_2)$. On the set $A_{(\xi,\mu,\tau)}$, it writes:
\begin{align*}
\H(\xi_1, \xi-\xi_1, \mu, \mu-\mu_1) = \omega(\xi, \mu)- \omega( \xi_1, \mu_1) -\omega(\xi- \xi_1, \mu-\mu_1).
\end{align*}
We thus get a bound on the measure of $A_{(\xi,\mu,\tau)}$:
\begin{align}\label{measure6}
\vert A_{(\xi,\mu,\tau)} \vert \lesssim \min(L_1,L_2) \vert B_{(\xi,\mu,\tau)} \vert, 
\end{align}
with
\begin{align*}
B_{(\xi,\mu,\tau)}:= \left\{ 
\begin{array}{l}
(\xi_1, \mu_1)\in \mathbb{R}^2 ; (\xi_1, \xi-\xi_1, \mu_1, \mu-\mu_1) \in S_1, \\
\vert (\xi_1, \mu_1) \vert \simeq N_1, \vert (\xi-\xi_1,\mu- \mu_1) \vert \simeq N_2 ,\\
\vert \tau-\omega(\xi,\mu)+ \H(\xi_1,\xi-\xi_1, \mu_1,\mu-\mu_1 ) \vert \lesssim \max(L_1,L_2)
\end{array}
\right\}.
\end{align*}

To bound the set $B_{(\xi,\mu,\tau)}$, we define first the set of projection on one variable:
\begin{align*}
B_{(\xi,\mu,\tau)}(\mu_1) := \left\{ \xi_1 \in \mathbb{R}; (\xi_1,\mu_1) \in B_{(\xi,\mu,\tau)}\right\},
\end{align*}
study the set by the variation of $\H$ along $\mu_1$, and conclude by a corollary of the intermediate value theorem:
\begin{lemm}[{\cite[Lemma 3.8]{MP}}]
Let $I$ and $J$ two intervals of $\mathbb{R}$, and $f:I\rightarrow J$ a $\mathcal{C}^1$ function. Then we get:
\begin{align} \label{IVT}
\left\vert \left\{ x \in I ; f(x) \in J \right\} \right\vert \le \frac{\vert J \vert}{\inf_{x \in I} \vert f'(x) \vert}.
\end{align}
\end{lemm}

Lets compute the variations of the function of resonance. The derivative of $\H$ can be bounded from below:
\begin{align} \label{derivative}
\left\vert \frac{d}{d\mu_1} \H(\xi_1, \xi-\xi_1, \mu_1, \mu-\mu_1) \right\vert = \left\vert 2\xi_1 \mu_1 -2(\xi-\xi_1)(\mu-\mu_1) \right\vert \gtrsim \max(N_1,N_2)^2 \simeq N_1^2.
\end{align}

We thus get a bound on the measure of $B_{(\xi,\mu,\tau)}$, by applying (\ref{IVT}) to (\ref{derivative}):
\begin{align}\label{measure7}
\vert B_{(\xi,\mu,\tau)} \vert \le \int_{\vert \mu_1 \vert \le 2N_1, \vert \mu- \mu_1 \vert \le 2N_2} \vert B_{(\xi,\mu,\tau)}(\mu_1) \vert d\mu_1 \lesssim \max(N_1,N_2) \frac{\max(L_1,L_2)}{N_1^2}\simeq \frac{\max(L_1,L_2)}{N_1}.
\end{align}

Gathering the bounds (\ref{measure6}) and (\ref{measure7}), we conclude the proof (\ref{measure5}).

The assumption of working in $S_1$ was used in (\ref{derivative}), where the sign condition was essential. 

For $J_2$, we work on $S_2$: the result is similar, except that we differentiate $\H$ in the $\xi_1$ direction:
\begin{align}\label{measure8}
\left\vert \frac{d}{d\xi_1} \H(\xi_1, \xi-\xi_1, \mu_1, \mu- \mu_1) \right\vert = \left\vert 3 \xi_1^2 + \mu_1^2 -3 (\xi-\xi_1)^2 -(\mu-\mu_1) \right\vert.
\end{align}

Without any assumption on the signs of the variables, we need to localize the different frequencies. The result is immediate from the following lemma:
\begin{lemm}[{\cite[Lemma 4.2]{MP}}]
Let $0<\alpha<1$. There exists a continuous function $f:[0,1] \rightarrow \mathbb{R}$ satisfying $f(\alpha) \rightarrow 0$ as $\alpha \rightarrow 0$, such that for all $(\xi_1,\mu_1)$ and $(\xi_2, \mu_2)$ satisfying:
\begin{align*}
\xi_1 \xi_2 <0, \quad \mu_1 \mu_2 <0, \quad (1-\alpha)^\frac{1}{2} \sqrt{3} \vert \xi_i \vert \le \vert \mu_i \vert \le (1-\alpha)^{-\frac{1}{2}} \sqrt{3}\vert \xi_i \vert,
\end{align*}
we obtain:
\begin{align*}
\vert (\xi_1 + \xi_2, \mu_1+ \mu_2) \vert^2 \le \left\vert \vert (\xi_1, \mu_1) \vert ^2 - \vert (\xi_2, \mu_2) \vert^2 \right\vert + f(\alpha) \max\left( \vert (\xi_1, \mu_1) \vert^2, \vert (\xi_2, \mu_2) \vert^2 \right).
\end{align*}
We recall that $\vert (\xi_1, \mu_1) \vert^2 = 3 \xi_1^2 +\mu_1^2 \simeq N_1^2$.
\end{lemm}

On $\mathcal{D}_{1,3}$ and $S_2$, the assumptions of the lemma are satisfied, and with $\alpha>0$ so small that $f(\alpha) \le \frac{1}{10}$, we deduce:
\begin{align*}
\frac{9}{10} \vert (\xi,\mu) \vert ^2 \le \left\vert \vert (\xi_1, \mu_1) \vert^2 - \vert (\xi-\xi_1, \mu-\mu_1) \vert^2 \right\vert, 
\end{align*}
and in the area of $high \times high \rightarrow high$ interaction, injecting in (\ref{measure8}):
\begin{align*}
\left\vert \frac{d}{d\xi_1} \H(\xi_1, \xi-\xi_1, \mu_1, \mu-\mu_1) \right\vert \gtrsim N^2 \simeq N_1^2.
\end{align*}
This estimate is equivalent to (\ref{derivative}), and concludes the proof of the bound (\ref{measure5}) in $S_2$.

\small
\bibliographystyle{plain}
\bibliography{biblio.bib}

\end{document}